\documentclass[12pt,reqno,a4paper]{amsart}
\usepackage{amssymb,latexsym}
\usepackage{amsthm, amsmath, amsfonts}
\usepackage{hyperref}
\usepackage{empheq}
\usepackage{chngcntr}
\usepackage{apptools}
\usepackage{graphicx}  % needed for figures
\usepackage{dcolumn}   % needed for some tables
\usepackage{bm}        % for math
\usepackage{enumerate}
\usepackage{slashed}
\usepackage{simplewick}
\usepackage{comment}
\usepackage{color}
\usepackage{soul}
\usepackage[english]{babel}
\usepackage{appendix}
\usepackage[utf8]{inputenc}
\usepackage{tikz} 
\usetikzlibrary{decorations.markings}
\usepackage{enumitem}
\numberwithin{equation}{section}

\usepackage{marginnote}
\usepackage[normalem]{ulem}

\DeclareMathOperator{\Td}{Td}
\DeclareMathOperator{\ch}{ch}
\DeclareMathOperator{\cc}{c}
\DeclareMathOperator{\Tr}{Tr}
\DeclareMathOperator{\ric}{Ric}
\DeclareMathOperator{\End}{End}

\newtheorem*{maindefn*}{{\sc Definition}}

\newtheorem{mainthm}{Theorem}

\theoremstyle{definition}
\newtheorem{definition}{Definition}

\newtheorem{theorem}{Theorem}
\newtheorem{lemma}{Lemma}
\newtheorem{proposition}{Proposition}

\theoremstyle{remark}
\newtheorem{remark}{Remark}

\title[Quillen metric and moment map]{Asymptotic expansion of the variation of the Quillen metric and its moment map interpretation}
\author{Kiyoon Eum}
\address{Department of Mathematical Sciences, KAIST, 291 Daehak-ro, Yuseong-gu, Daejeon 34141, South Korea}
\email{kyeum@kaist.ac.kr}

\begin{document}

\begin{abstract}
In K\"ahler geometry, the Donaldson--Fujiki moment map picture interprets the scalar curvature of a K\"ahler metric as a moment map on the space of compatible almost complex structures on a fixed symplectic manifold. In this paper, we generalize this picture using the framework of equivariant determinant line bundles. Given a prequantization $P=(L,h,\nabla)$ of a compact symplectic manifold $(M,\omega)$, let $\mathcal{G}=\mathrm{Aut}(P)$. For each $k\in\mathbb{N}$, we construct a $\mathcal{G}$-equivariant determinant line bundle $\lambda^{(k)}\rightarrow\mathcal{J}_{int}$ on the space of integrable compatible almost complex structures, equipped with the $\mathcal{G}$-invariant Quillen metric. The curvature form of $\lambda^{(k)}$ admits an asymptotic expansion whose coefficients yield a sequence of $\mathcal{G}$-invariant closed $2$-forms $\Omega_j$ on $\mathcal{J}_{int}$ and corresponding moment maps $\mu_j:\mathcal{J}_{int}\rightarrow C^\infty(M)$. Each $\mu_j$ arises from the asymptotic expansion of the variation of the logarithm of the Quillen metric with respect to K\"ahler potentials, with the complex structure held fixed. This provides a natural generalization of the Donaldson--Fujiki moment map interpretation of scalar curvature. Moreover, we show that $\mu_j$ coincide with the $Z$--critical equations introduced by Dervan--Hallam, and we state a generalization of Fujiki's fiber integral formula.
\end{abstract}

\maketitle

\thispagestyle{empty}
\tableofcontents

%%%%%%%%%%%%%%%%%%%%%%%%%%%%%%%%

\section{Introduction}

In K\"ahler geometry, the Donaldson--Fujiki moment map picture \cite{donaldson1997remarks,fujiki1992moduli} characterizes the scalar curvature of a K\"ahler metric as a moment map on the space of compatible almost complex structures on a fixed symplectic manifold. In light of the Kempf--Ness theorem, the Yau--Tian--Donaldson conjecture proposes a correspondence between the existence of a \textit{constant scalar curvature K\"ahler (cscK) metric} and an algebro-geometric notion of stability. Since its formulation, it has become an important topic in K\"ahler geometry. For recent developments, see \cite{dervan2025stability} and the references therein.

The original proof of the Donaldson–Fujiki moment map picture was obtained by direct computation, and it is difficult to generalize this approach to more complicated curvature quantities. However, it was already pointed out in \cite{fujiki1992moduli,donaldson1997remarks,donaldson2001planck} that the moment map structure should be related to an equivariant line bundle on the space of compatible almost complex structures, which is naturally expected to be the \textit{determinant line bundle} developed by Quillen \cite{quillen1985determinants} and Bismut--Gillet--Soul\'e \cite{bismut1988analytic3}. In this paper, we use the equivariant determinant line bundle to derive the moment map property for scalar curvature-type quantities associated with K\"ahler metrics, generalizing the Donaldson--Fujiki moment map picture for scalar curvature. Note that a similar approach was taken in \cite{foth2007manifold}; we will explain how our result relates to theirs.

Let us describe our results. Let $(M,J,\omega)$ be a compact K\"ahler manifold. Fix the symplectic manifold $(M,\omega)$ and let $\mathcal{J}_{int}$ be the space of \textit{integrable compatible almost complex structures} on $M$. The set $\mathcal{J}_{int}$ has a natural structure of a complex manifold (see Section \ref{spaceof} for details). Assume $[\omega]$ is integral and fix a \textit{prequantization} $P=(L,h,\nabla)$ of $(M,\omega)$. Then $\mathcal{G}:=\mathrm{Aut}(P)$ acts on $\mathcal{J}_{int}$ by biholomorphisms. We consider a \textit{universal family} over $\mathcal{J}_{int}$, whose fiber over $J$ is the complex manifold $M_J$ endowed with the complex structure $J$. 

For each $k\in\mathbb{N}$, the construction of \cite{bismut1988analytic3} provides a determinant line bundle $\lambda^{(k)}$ on $\mathcal{J}_{int}$, constructed from $L^k$. The group $\mathcal{G}$ acts equivariantly on $\lambda^{(k)}\rightarrow\mathcal{J}_{int}$ and preserves the Chern connection associated with the \textit{Quillen metric} $||\cdot||$. From this geometric setup, for each $k\in \mathbb{N}$ we can derive the moment map corresponding to the Chern curvature form $\Omega^{(k)}$ of $\lambda^{(k)}$. By taking the $k\rightarrow\infty$ limit and comparing the coefficients in the asymptotic expansion, we obtain a sequence of scalar quantities $\mu_j$ on K\"ahler metrics, which have a natural interpretation as moment maps. 

The quantities $\mu_j$ are obtained by differentiating the logarithm of the Quillen metric with respect to the K\"ahler potential, while keeping the complex structure $J$ fixed. This is natural from the viewpoint of the Kempf--Ness theorem, where the moment map is obtained by differentiating the log-norm functional in a complexified orbit direction. In our case, a genuine complexification of $\mathcal{G}$ does not exist, but varying the K\"ahler potential while fixing $J$ can be interpreted as moving along a complexified orbit. That is, we obtain the following theorem. For precise definitions of the objects appearing in the statement, see Sections \ref{Zcrit} and \ref{momentproof}.

\begin{mainthm}\label{thm:mainA}
Fix a complex structure $J$ and consider the derivative of the logarithm of the Quillen metric with respect to the K\"ahler potential $\varphi$. Taking the ratio with the Quillen metric associated with the fixed K\"ahler potential, we may regard it as a scalar. Its asymptotic expansion as $k\rightarrow \infty$ has the following form:
\begin{equation*}
\frac{\delta}{\delta\varphi}\log \frac{||\cdot||^2_{\omega_\varphi,h^ke^{-k\varphi}}}{||\cdot||^2_{\omega,h^k}}=\sum_{j=0}^{n+1} k^{n+1-j} \int_M \delta\varphi \, \mu_j\left(g_{(\omega_\varphi,J)}\right) \,\frac{\omega_{\varphi}^n}{n!}
\end{equation*}
for some functions $\mu_j\in C^\infty(M)$ which depend on the curvature of the K\"ahler metric $g_{(\omega_\varphi,J)}$.
Then the map $\mu_j:\mathcal{J}_{int}\rightarrow C^\infty(M)$ defined by 
\begin{equation*}
\mu_j(J'):=\mu_j\left(g_{(\omega,J')}\right)
\end{equation*}
is a moment map for the action $\mathcal{G}\curvearrowright (\mathcal{J}_{int},\Omega_j)$, where $\Omega_j\in \mathcal{A}^2(\mathcal{J}_{int})$ is given by the following fiber integral
\begin{equation*}
\Omega_j=\int_{\mathcal{X}/\mathcal{J}_{int}}\Td_j(TM,g_{(\omega,J)})\ch_{n+1-j}(\widetilde{L},\widetilde{h}).
\end{equation*}
\end{mainthm}

Note that the asymptotic expansion of $\log||\cdot||^2_{\omega_\varphi,h^ke^{-k\varphi}}$ naturally provides, via its expansion coefficients, a sequence of functionals on the space of K\"ahler potentials, and, upon differentiation in the direction of holomorphic vector fields, yields higher Futaki invariants \cite{futaki2004asymptotic}; see \cite[Section 4]{eum2025partition}. In \cite{wernerwendland}, critical points of the logarithm of the Quillen metric corresponding to a certain virtual bundle are identified with cscK metrics. Here, instead, we consider the sequence of Quillen metrics for each $k$, and cscK metrics, together with higher-order generalizations, appear as critical points of the coefficients in the large $k$ asymptotic expansion of the logarithm of the Quillen metric.

In \cite{dervan2023universal}, Dervan--Hallam introduced the notion of \textit{Z--critical equations} on K\"ahler metrics, which serve as the analytic counterpart of $Z$--stability \cite{dervan2023stability}, modeled on Bridgeland stability conditions. We will show that $\mu_j$ are examples of $Z$--critical equations with central charge given by the Todd class. In \cite{dervan2023universal}, Dervan--Hallam also established a stronger moment map property of $Z$--critical equations, in particular including the $\mu_j$, by a different method. We will explain in Section \ref{spaceof} in what sense their result is stronger. 

For $j=1,2$, we have explicit formulas for $\mu_j$; see \eqref{expformula}. Together with Fujiki's fiber integral formula \cite[Theorem 4.4]{fujiki1992moduli}, this shows that the case $j=1$ is precisely the Donaldson--Fujiki moment map picture of scalar curvature in the integrable case. We also note that the result of \cite{foth2007manifold} can be reinterpreted as the following generalization of Fujiki’s fiber integral formula. Here, $a_j(\omega)$ denotes the $j^{th}$ coefficient in the TYZ expansion of the Bergman kernel, and $\beta_j$ denotes the $j^{th}$ local coefficient in the asymptotic expansion of the Ray--Singer torsion \cite{finski2018full}. See Section \ref{Zcrit} for further exposition.

\begin{mainthm}\label{thm:mainB}

There are functions $s_r(A,B)\in C^\infty(M)$ for $A, B\in T_J \mathcal{J}_{int}$ such that for each $j\geq 1$, as a $2$-form on $\mathcal{J}_{int}$,
\begin{equation*}
\int_{\mathcal{X}/\mathcal{J}_{int}}\Td_j(TM,g_{(\omega,J)})\ch_{n+1-j}(\widetilde{L},\widetilde{h})(A,B)=-\int_M \sum_{r=0}^{j-1} s_r(A,B) a_{j-1-r}(g_{(\omega,J)})\frac{\omega^n}{n!}-i\partial\bar{\partial}\beta_{j-1}(A,B).
\end{equation*}
In particular, $s_0(A,B)=\frac{1}{8}\operatorname{Tr}_{TM}(JAB)$.
\end{mainthm}
From the explicit expressions for $\beta_0, \beta_1$ \cite{vasserot1989asymptotics,finski2018full}, it follows that $\partial\bar{\partial}\beta_{j-1}=0$ for $j=1,2$. In particular, the $j=1$ case coincides with \cite[Theorem 4.4]{fujiki1992moduli}.

So far, the only explicitly computed $Z$--critical equations have been those corresponding to powers of the first Chern class of $M$ \cite[Section 2.2]{dervan2023stability}. In Appendix \ref{A}, we introduce a simple trick and compute the case corresponding to $\ch_2$, which might be useful analytically.

\begin{remark}[Remark on deformation quantization]\label{deforquant}
Substituting $k=\hbar^{-1}$ in Theorem \ref{thm:mainA} and its proof, one obtains an equation of the form
\begin{equation}\label{qm}
    d\langle\mu, f\rangle=\iota_{\widehat{f}}\Omega
\end{equation}
where 
\begin{equation*}
\langle\mu, f\rangle=\sum_{j=0}^{\infty}\hbar^{j}\langle\mu_{j+1},f\rangle \in C^{\infty}(\mathcal{J}_{int})[[\hbar]]
\end{equation*}
is a formal function, $\Omega=\sum_{j=0}^{\infty}\hbar^j \Omega_{j+1}$ is a formal deformation of (the minus of) the Donaldson-Fujiki K\"ahler form $\Omega_1$, and $f\in C^{\infty}(M)\simeq \mathrm{Lie}(\mathcal{G})$. Let $\widetilde{\mu}:C^\infty(M)\rightarrow C^\infty(\mathcal{J}_{int})$ be defined by $\widetilde{\mu}(f):=\langle\mu, f\rangle$. It is known \cite[Deduction 4.1]{muller2004invariant} that equation (\ref{qm}) is equivalent to $\widetilde{\mu}$ being a \textit{quantum Hamiltonian} for the action of $\mathcal{G}$ on the invariant star product of Wick type with Karabegov form $\Omega$. It is also easy to see that $\widetilde{\mu}$ is formally a \textit{quantum moment map}; see \cite{xu1998fedosov,muller2004some,gutt2003natural,futaki2021quantum} and \cite{la2021formal,la2022scalar} for related results on deformation quantization. We also note that a similar deformation of the K\"ahler form appears in
\cite[(9.4)]{takhtajan2006quantum}, and it would be interesting to understand its relation to
$\Omega$.
\end{remark}

\subsection*{Acknowledgements}
The author would like to thank Ruadhaí Dervan for helpful comments concerning \cite{dervan2023universal} and for kindly sharing a recent version of the manuscript prior to publication. This work was supported by the National Research Foundation of Korea (NRF) grant funded by the Korea government (MSIT) RS-2024-00346651.

\section{Preliminaries}

\subsection{Moment maps and equivariant line bundles}\label{equivline}

The conventions and notation in this subsection follow those of \cite[Section 7.1]{BerlineGetzlerVergne2004}. Let $M$ be a symplectic manifold with symplectic 2-form $\Omega$. For $f\in C^\infty(M)$, the \textit{Hamiltonian vector field} defined by $f$ and $\Omega$ is the unique vector field $X_f$ such that
\begin{equation}\label{hamilton}
df=\iota_{X_f}\Omega.
\end{equation}
Assume that a Lie group $G$ acts on $M$ by $(g,x)\mapsto g.x$ for $g\in G, x\in M$. For $X\in \mathfrak{g}:=\mathrm{Lie}(G)$, we define the \textit{fundamental vector field} $\widehat{X}$ on $M$ as
\begin{equation*}
\widehat{X}_x:=\left.\frac{d}{dt}\right|_{t=0}\exp(-tX).x \in T_x M.
\end{equation*}
Then the map $\mathfrak{g}\ni X\mapsto \widehat{X}\in \Gamma (TM)$ is a Lie algebra homomorphism.

\begin{definition}
Let $G\curvearrowright (M,\Omega)$ be a symplectic action. We say that the action is \textit{Hamiltonian} if there is a \textit{moment map}
\begin{equation*}
\mu : M \rightarrow \mathfrak{g}^*
\end{equation*}
satisfying two conditions:
\begin{enumerate}
    \item Equivariance: $\quad \forall g\in G, \forall x\in M, \forall X\in \mathfrak{g},\quad \langle \mu_{g.x}, X\rangle=\langle \mu_{x}, \rm{Ad}_{g^{-1}}X\rangle$;
    \item Moment map equation: $\forall X\in \mathfrak{g},\quad d(x\mapsto \langle \mu_x,X\rangle)=\iota_{\widehat{X}}\Omega$.
\end{enumerate}
\end{definition}
One situation in which such a moment map can be derived is the following. Suppose there exists a $G$-equivariant complex line bundle $L$ on $M$ endowed with a $G$-invariant connection $\nabla^{L}$ whose curvature is
\begin{equation*}
\left(\nabla^L \right)^2=-i\Omega \in \mathcal{A}^2(M).
\end{equation*}
The group $G$ acts on the space of sections $\Gamma(M,L)$ by the formula
\begin{equation}\label{pullback}
(g.s)(x)=g^{L}.(s(g^{-1}.x)).
\end{equation}
For $X\in\mathfrak{g}$, we denote by $\mathcal{L}^L(X)$ the corresponding infinitesimal action on $\Gamma(M,L)$,
\begin{equation*}
\mathcal{L}^L(X)s=\left.\frac{d}{dt}\right|_{t=0}\exp(tX).s
\end{equation*}
which is called the \textit{Lie derivative}. For each $X$, define
\begin{equation}\label{moment}
\widetilde{\mu}(X)=\mathcal{L}^L(X)-\nabla^L_{\widehat{X}} \in \Gamma(M,\End{L})\simeq C^\infty(M,\mathbb{C}).
\end{equation}
Then, by \cite[Example 7.9]{BerlineGetzlerVergne2004}, the map $\mu : M\rightarrow \mathfrak{g}^*$ defined by $\langle \mu_x,X\rangle=i\widetilde{\mu}(X)$ is the moment map as defined above. That is,
\begin{equation*}
d\bigl(i\widetilde{\mu}(X)\bigr)=i\iota_{\widehat{X}}\left(\nabla^L \right)^2=\iota_{\widehat{X}}\Omega.
\end{equation*}
Note that, in this case, the nondegeneracy condition on $\Omega$ is not essential, and we still call such a $\mu$ a moment map even without the nondegeneracy condition (see \cite{dervan2023universal} for a similar convention).

Now consider the holomorphic case. Suppose $(M,J,\Omega)$ is a K\"ahler manifold and $(L,h)$ is a Hermitian holomorphic line bundle on $M$ whose Chern curvature form is $-i\Omega$. If $G$ acts equivariantly on $(L,h)\rightarrow M$ by biholomorphisms preserving $h$, we can apply the previous construction of the moment map to the Chern connection. Let $\nabla^h$ be the Chern connection compatible with $h$, and let $\mu^h$ be the moment map obtained from the previous construction. If $g\in C^\infty(M)^G$ is a $G$-invariant smooth function on $M$, then the $G$-action also preserves $he^{-g}$, and we obtain a moment map for the new $2$-form $\Omega+i\partial\bar{\partial}g$. The difference between the two moment maps is given by
\begin{equation}\label{compare}
\widetilde{\mu}^{he^{-g}}(X)-\widetilde{\mu}^h(X)=-\nabla^{he^{-g}}_{\widehat{X}}+\nabla^{h}_{\widehat{X}}=\iota_{\widehat{X}} \partial g\in C^\infty(M,\mathbb{C}).
\end{equation}

\subsection{Determinant line bundles and Quillen metrics}\label{bgs}

We review the construction of the determinant line bundle and the Quillen metric following \cite{bismut1988analytic3}. Let $\pi : \mathcal{X}\rightarrow B$ be a proper holomorphic submersion of relative dimension $n$ with connected fibers $M_y:=\pi^{-1}(y)$. Let $TM$ be the relative (vertical) tangent bundle and $T^H \mathcal{X}$ be a horizontal subbundle of $T\mathcal{X}$ such that $T\mathcal{X}=T^H \mathcal{X} \oplus TM$. For $Y\in TB, Y^H\in T^H \mathcal{X}$ is the horizontal lift of $Y$ in $T\mathcal{X}$ so that $\pi_*(Y^H)=Y$. Let $g^M$ be a smooth family of K\"ahler metrics on the fibers. The triple $(\pi,g^M,T^H\mathcal{X})$ is called a \textit{K\"ahler fibration} if there exists a closed smooth $(1,1)$-form $\omega$ on $\mathcal{X}$ such that $T^H \mathcal{X}$ and $TM$ are orthogonal with respect to $\omega$ and the restriction $\omega|_{M_y}$ of $\omega$ to each fiber defines a K\"ahler metric $g^M$. We say that $\omega$ is associated with $(\pi,g^M,T^H\mathcal{X})$. Note that this notion is equivalent to the notion of a \textit{refined metrically polarized family} in \cite[Remark 3.2]{fujiki1990moduli}.

Assume $(\pi,g^M,T^H\mathcal{X})$ is a K\"ahler fibration. Let $\xi$ be a Hermitian holomorphic vector bundle on $\mathcal{X}$. For $0\leq p\leq n$, $E^p$ denotes the set of smooth sections over $\mathcal{X}$ of $\Lambda^p(T^{*(0,1)}M)\otimes\xi$. For each $y\in B$, the fiber $E^p_y$ is the set of smooth sections over $M_y$ of $\Lambda^p(T^{*(0,1)}M)\otimes\xi$. Set $E=\oplus_p E^p$. For each $y\in B$, $E_y$ is endowed with the Hermitian product
\begin{equation*}
s,s'\in E_y \mapsto \int_{M_y}\langle s,s'\rangle_{g^M,h^{\xi}}(x)d\mathrm{vol}_{g^M}.
\end{equation*}
Let $\nabla$ be the Chern connection on the Hermitian holomorphic vector bundle $\Lambda(T^{*(0,1)}M)\otimes\xi$ on $\mathcal{X}$. For $0\leq p \leq n$, let $\widetilde{\nabla}$ be the connection on $E^p$ such that if $s$ is a smooth section of $E^p$ and if $Y\in TB$,
\begin{equation*}
\widetilde{\nabla}_Y s = \nabla_{Y^H} s.
\end{equation*}
Then $\widetilde{\nabla}$ preserves the Hermitian product of $E$ \cite[Theorem 1.14]{bismut1988analytic2}.

Define $D_y=\bar{\partial}_y+{\bar{\partial}_y}^*$ acting on $E_y$. Then $\mathrm{Spec}(D_y^2)$ is discrete for all $y\in B$. Let $K^{b,p}_y$ be the sum of the eigenspaces of the operator $D^2_y$ acting on $E^p_y$ with eigenvalues $<b$. Let $U^b$ be the open set $U^b=\{ y\in B : b\notin \mathrm{Spec}(D_y^2) \}$. On $U^b$, define the line bundle \begin{equation*}
\lambda^b=\otimes_p (\det K^{b,p})^{(-1)^{p+1}}.
\end{equation*}
Then $\lambda^b, b>0$ patch into the $C^\infty$ line bundle $\lambda$ over $B$. As a smooth subbundle of $E$ over $U^b$, $K^b=\oplus_p K^{b,p}$ inherits the Hermitian product, which induces a smooth Hermitian metric $|\cdot|^b$ on $\lambda^b$. For $y\in U^b$, let $P^b_y$ be the orthogonal projection operator from $E_y$ onto $K^b_y$. Define the connection $\nabla^b$ on $K^b$ over $U^b$ by
\begin{equation*}
\nabla^b s=P^b \widetilde{\nabla} s.
\end{equation*}
Then $\nabla^b$ preserves the metric on $K^b$ and induces a connection ${}^0\nabla^b$ on $\lambda^b$ compatible with $|\cdot|^b$. Define a holomorphic structure on $\lambda^b$ over $U^b$ so that ${}^0\nabla^b$ is the Chern connection for $(\lambda^b,|\cdot|^b)$. Such holomorphic structures on $(\lambda^b,U^b)$ patch into a uniquely defined holomorphic structure on $\lambda$ over $B$ \cite[Theorem 1.3]{bismut1988analytic3}.

For $y\in U^b$, let $\tau_y(\bar{\partial}^{(b,\infty)})$ be the Ray-Singer analytic torsion \cite{ray1973analytic} defined in \cite[Definition 1.5]{bismut1988analytic3}. Let $||\cdot||^b=|\cdot|^b \tau_y(\bar{\partial}^{(b,\infty)})$ be the metric. Let 
\begin{equation*}
{}^1\nabla^b={}^0\nabla^b+\partial^B\log \tau^2(\bar{\partial}^{(b,\infty)})
\end{equation*}
be the connection on the line bundle $\lambda^b$ on $U^b$. Then the metrics $||\cdot||^b$ patch into the smooth metric $||\cdot||$ on $\lambda$, called the \textit{Quillen metric}, and the connections ${}^1\nabla^b$ patch into the Chern connection for $(\lambda,||\cdot||)$ \cite[Theorem 1.6]{bismut1988analytic3}. The curvature of ${}^1\nabla$ is computed as follows.

\begin{theorem}[BGS curvature formula, \protect{\cite[Theorem 1.9]{bismut1988analytic3}}]
The curvature of $^1\nabla$ is given by
\begin{equation}\label{Qcurv}
\left( {}^1\nabla\right)^2=i\left[\int_{\mathcal{X}/B}\Td(TM,g^M)\ch(\xi,h^{\xi})\right]^{(2)},
\end{equation}
where $[\cdot]^{(2)}$ denotes the component of degree $2$.
\end{theorem}
The same formula also holds for \textit{locally K\"ahler fibrations} \cite[Theorem 1.27]{bismut1988analytic3}.

\subsection{Space of compatible almost complex structures}\label{spaceof}

We follow the conventions of \cite[Section 1.3]{scarpa2020hitchin}, except for the opposite sign convention in (\ref{hamilton}). Fix a compact complex manifold $M$ and a K\"ahler form $\omega$, and let $\mathcal{J}$ be the set of \textit{compatible almost complex structures} on $M$, i.e.
\begin{equation*}
\mathcal{J}=\{ J\in \Gamma(M,\End TM):J^2=-\mathrm{Id},\; \omega(J\cdot,J\cdot)=\omega(\cdot,\cdot),\; \omega(\cdot,J\cdot)>0 \}.
\end{equation*}
For $J\in\mathcal{J}$, denote by $g_{(\omega,J)}$ the Riemannian metric on $M$ defined by $\omega$ and $J$:
\begin{equation*}
g_{(\omega,J)}(\cdot,\cdot)=\omega(\cdot,J\cdot).
\end{equation*}
The set $\mathcal{J}$ has the natural structure of an infinite-dimensional Fr\'echet manifold. The tangent space of $\mathcal{J}$ at $J$ is given by
\begin{equation*}
T_J \mathcal{J}=\{A\in \Gamma(M,\End TM):AJ+JA=0,\;\omega(A\cdot,J\cdot)+\omega(J\cdot,A\cdot)=0 \}.
\end{equation*}
Define the almost complex structure $\mathbb{J}$ on $\mathcal{J}$ by setting
\begin{align*}
\mathbb{J}_J : T_J \mathcal{J}&\rightarrow T_J \mathcal{J}\\
 A &\mapsto JA.
\end{align*}
This gives $\mathcal{J}$ the structure of a \textit{complex} Fr\'echet manifold \cite{koiso1990complex}.

The group $\mathrm{Ham}(M,\omega)$ of Hamiltonian diffeomorphisms of $M$ acts naturally on $\mathcal{J}$ by pullback: for $\psi\in \mathrm{Ham}(M,\omega), J\in \mathcal{J}$,
\begin{equation}\label{actofham}
\psi.J:=(\psi^{-1})^* J=\psi_*\circ J \circ \psi^{-1}_*. 
\end{equation}
The Lie algebra of $\mathrm{Ham}(M,\omega)$ can be identified with $C^\infty_0(M)$ via $-X_f \simeq f$ (\ref{hamilton}). The fundamental vector field on $\mathcal{J}$ of a function $f\in C^\infty_0(M)$ is then given by
\begin{equation}\label{infaction}
\widehat{f}_J=\left.\frac{d}{dt}\right|_{t=0}\exp(-tf).J=\left.\frac{d}{dt}\right|_{t=0}\left(\Phi^{t}_{-X_f} \right)^*J=\mathcal{L}_{-X_f}J,
\end{equation}
where $\Phi^t_{-X_f}$ is the flow of the Hamiltonian vector field $-X_f$. One can check that $\mathcal{L}_{\widehat{f}}\mathbb{J}=0$ for all $f$, that is, $\mathrm{Ham}(M,\omega)$ acts on $\mathcal{J}$ by biholomorphisms.

Let $\mathcal{J}_{int}$ denote the space of \textit{integrable} compatible almost complex structures, $\mathcal{J}_{int}=\{J\in\mathcal{J}: N_J=0\}$, where $N_J$ is the Nijenhuis tensor of $J$. For $J\in\mathcal{J}_{int}$, $(g_{(\omega,J)},J,\omega)$ is a K\"ahler triple. With respect to $\mathbb{J}$, $\mathcal{J}_{int}$ is a complex submanifold of $\mathcal{J}$ on its smooth locus. 

In fact, it is an analytic subset of $\mathcal{J}$ because of the possible presence of singularities; however, we do not address this issue in the present paper. Here, we restrict attention to finite-dimensional complex submanifolds of $\mathcal{J}_{int}$, and adopt the convention that $\mu$ is a moment map on $\mathcal{J}_{int}$ if it is a moment map on every such submanifold. This is weaker than requiring $\mu$ to be the restriction of a moment map on the full space of almost complex structures, as proved in \cite{dervan2023universal}, which generalizes \cite{donaldson1997remarks}. We emphasize that this choice is not merely a technical restriction: in the proof of the moment map equation, specifically in (\ref{integrability}), we explicitly use the integrability of the almost complex structure $J$.

We now construct a universal family over $\mathcal{J}_{int}$. Let $M_J:=(M,J)$ be the complex manifold corresponding to $J\in\mathcal{J}_{int}$. Then the natural map $\pi : \mathcal{X} = \bigcup_J M_J \rightarrow \mathcal{J}_{int}$ gives the \textit{universal family} of K\"ahler manifolds $(M_J,g_{(\omega,J)})$ depending holomorphically on $J\in\mathcal{J}_{int}$. More precisely, we define $\mathcal{X}:=M\times\mathcal{J}_{int}$ with complex structure
\begin{equation*}
\mathbb{J}^{\mathcal{X}}_{(x,J)}(X,A):=(J_x X,\mathbb{J}_J A)
\end{equation*}
for $X\in T_x M$ and $A\in T_J \mathcal{J}_{int}$. It can be verified that this almost complex structure is integrable.
 
Assume that the de Rham class $[\omega]$ of $\omega$ is integral. Then there is a triple $P=(L,h,\nabla)$, called a \textit{prequantization}, consisting of a complex line bundle $L$ on $M$, a Hermitian metric $h$ on $L$, and a connection $\nabla$ compatible with $h$, such that $\nabla^2=-i\omega$. Fix a prequantization $P$ of $(M,\omega)$. For each $J\in\mathcal{J}_{int}$, the pair $(J,\nabla)$ uniquely determines a holomorphic structure on $L$ so that $\nabla$ becomes the Chern connection of the resulting holomorphic line bundle $(L,h)$ on $M_J$. We denote this structure by $(L_J,h_J)$. These define a Hermitian holomorphic line bundle $(\widetilde{L},\widetilde{h})$ on $\mathcal{X}$, where $\widetilde{L}=\bigcup_J L_J, \widetilde{h}=\{h_J\}$. In this case, $(\pi:\mathcal{X}\rightarrow \mathcal{J}_{int}, \cc_1(\widetilde{L},\widetilde{h}), \widetilde{L})$ is a \textit{metrically polarized family of Hodge manifolds} in the sense of \cite[Definition 3.8]{fujiki1990moduli}. In particular, it is a K\"ahler fibration, as is clear from the construction of $\mathcal{X}$; see also \cite[Remark 3.4]{fujiki1990moduli}.

Let $\mathcal{G}=\mathrm{Aut}(P)$ be the group of bundle diffeomorphisms of $L$ that preserve $h$ and $\nabla$. Then any element of $\mathcal{G}$ induces an element of $\mathrm{Ham}(M,\omega)$ on the base $M$, and acts on $\mathcal{J}_{int}$ by (\ref{actofham}). We will prove the moment map property for the action of $\mathcal{G}$ on $\mathcal{J}_{int}$ with respect to the appropriate closed $\mathcal{G}$-invariant $2$-forms $\Omega_j$ on $\mathcal{J}_{int}$.

\section{Main results}

\subsection{Variation of the Quillen metric and the Z--critical equations}\label{Zcrit}

In this section we fix a complex structure $J$ on $M$ and consider the variation of K\"ahler metrics within the K\"ahler class $[\omega]$. Let $\mathcal{K}_\omega=\{\varphi\in C^\infty(M):\omega_\varphi=\omega+i\partial\bar{\partial}\varphi>0 \}$ be the space of K\"ahler potentials. Our moment maps will be derived by taking the derivative of the asymptotic expansion of the Quillen metric with respect to the K\"ahler potential $\varphi$. We start by recalling the results of \cite{eum2025partition}; see \textit{ibid.} for further details.

Let $(M,\omega)$ be a polarized K\"ahler manifold with a polarizing line bundle $L$, that is, $\cc_1(L)=[\omega]$. For each K\"ahler metric $\omega_\varphi\in [\omega]$, there exists a Hermitian metric $h_\varphi$ on $L$ such that $(\nabla^{h_\varphi})^2=-i\omega_\varphi$. For example, it can be given by $h_\varphi=he^{-\varphi}$, which is determined up to multiplication by a constant.

Let $\lambda(L^k)$ be the determinant line constructed in Section \ref{bgs} with $\xi=L^k$, $B=\{\mathrm{point}\}$. Let $||\cdot||_{\omega,h^k}$ and $||\cdot||_{\omega_\varphi,h^ke^{-k\varphi}}$ be the Quillen metrics constructed from the data $g^M=g_\omega, h^\xi=h^k$ and $g^M=g_{\omega_\varphi}, h^\xi=h^ke^{-k\varphi}$, respectively. We compute the asymptotic expansion of $\log ||\cdot||_{\omega,h^k}-\log ||\cdot||_{\omega_\varphi,h^ke^{-k\varphi}}$ as $k\rightarrow \infty$ by two different methods. 

\textit{Method 1:} By construction, one can obtain the asymptotic expansion from the asymptotic expansions of (a) the $L^2$-product on $\det H^0(M,L^k)$ and (b) the Ray-Singer analytic torsion $\tau_{(\omega,h^k)}:=\tau_{(\omega,h^k)}(\bar{\partial}^{(0,\infty)})$. The asymptotic expansions of the $L^2$-products can be obtained from Donaldson's lemma and the TYZ expansion of the Bergman kernel (see, for example, \cite{ma2007holomorphic} for a comprehensive exposition). Denote by $a_j(\omega)$ the $j^{th}$ coefficient in the TYZ expansion, that is,
\begin{equation*}
\rho_k(\omega)=\sum_{j=0}^\infty k^{n-j} a_j(g_\omega) 
\end{equation*}
where $\rho_k(\omega)$ is the diagonal of the $k^{th}$ Bergman kernel associated with $(\omega,h^k)$. The coefficients $a_j(g_\omega)$ are smooth functions given by universal polynomials in the curvature of $g_\omega$ and its derivatives. Moreover, in the case where $M$ is the fiber of a proper holomorphic submersion, the coefficients are smooth over the base of the family, and derivatives over the base commute with the asymptotics \cite[Theorem 1.6]{ma2023superconnection}. Then, by \cite[Lemma 2]{donaldson2005scalar}, we get
\begin{equation*}
2\log \frac{|\cdot|_{\omega,h^k}}{|\cdot|_{\omega_\varphi,h^ke^{-k\varphi}}}=\sum_{j=0}^\infty k^{n+1-j}\int_0^1 dt\int_M \dot{\varphi_t}(\Delta_t a_{j-1}(g_t)-a_{j}(g_t))\frac{\omega_t^n}{n!}
\end{equation*}
for any smooth path $\varphi_t$ in $\mathcal{K}_\omega$ joining $0$ to $\varphi$. We use a subscript to denote objects associated with $\omega_t, \omega_\varphi$, etc. When no subscript is used, the object is understood to be associated with $\omega$.

For the Ray--Singer analytic torsion $\tau_{(\omega,h^k)}$, Finski established the full asymptotic expansion generalizing the result of \cite{vasserot1989asymptotics}:
\begin{theorem}[\protect{\cite[Theorem 1.1, Remark 1.2]{finski2018full}}]\label{Finski}
There are local coefficients $\alpha_j, \beta_j\in\mathbb{R}$ for $j\in\mathbb{N}$, such that as $k\rightarrow\infty$,
\begin{equation}\label{fintor}
-2\log \tau_{(\omega,h^k)}=\sum_{j=0}^{\infty}k^{n-j}(\alpha_j\log k + \beta_j(g_\omega)).
\end{equation}
The coefficients $\alpha_j$ do not depend on $\omega$. Moreover, in the case where $M$ is the fiber of a proper holomorphic submersion, 
the coefficients are smooth over the base of the family, and derivatives 
over the base commute with the asymptotics (\ref{fintor}).
\end{theorem}
Here, the local coefficient means that it can be expressed as an integral of a density defined locally over $M$, given by $g_\omega$ and its derivatives. Combining these, we obtain
\begin{equation}\label{exp1}
2\log \frac{||\cdot||_{\omega,h^k}}{||\cdot||_{\omega_\varphi,h^ke^{-k\varphi}}}=\sum_{j=0}^\infty k^{n+1-j}\left(\int_0^1 dt\int_M \dot{\varphi_t}(\Delta_t a_{j-1}(g_{t})-a_{j}(g_t))\frac{\omega_t^n}{n!}+\beta_{j-1}(g_\varphi)-\beta_{j-1}(g)\right)
\end{equation}
as $k\rightarrow\infty$.

\textit{Method 2:} Alternatively, one can derive the asymptotic expansion (in fact, an equality for each $k$) from the \textit{Quillen anomaly formula} \cite[Theorem 1.23]{bismut1988analytic3}, obtaining
\begin{align}\label{exp2}
&2\log \frac{||\cdot||_{\omega,h^k}}{||\cdot||_{\omega_\varphi,h^ke^{-k\varphi}}}\nonumber\\
&=\sum_{j=0}^{n+1} k^{n+1-j}\int_M -iBC(\Td_j;g_{\varphi},g)\frac{\omega^{n+1-j}}{(n+1-j)!}+\Td_{j}(R_\varphi)\frac{-1}{(n+1-j)!}\sum^{n-j}_{s=0}\varphi \omega_{\varphi}^s\wedge \omega^{n-j-s}
\end{align}
where $BC$ denotes \textit{Bott--Chern forms}; see \cite{BC,bismut1988analytic1}.
By comparing coefficients of (\ref{exp1}) and (\ref{exp2}), we have for all $j\geq0$,
\begin{align}\label{exp3}
&\int_0^1 dt\int_M \dot{\varphi_t}(a_{j}(g_t)-\Delta_t a_{j-1}(g_t))\frac{\omega_t^n}{n!}+\beta_{j-1}(g)-\beta_{j-1}(g_\varphi)\nonumber\\
&=\int_M iBC(\Td_j;g_{\varphi},g)\frac{\omega^{n+1-j}}{(n+1-j)!}+\Td_{j}(R_\varphi)\frac{1}{(n+1-j)!}\sum^{n-j}_{s=0}\varphi \omega_{\varphi}^s\wedge \omega^{n-j-s}.
\end{align}

Our moment maps will be derived from taking the derivative of (\ref{exp3}) with respect to the K\"ahler potential $\varphi$. Since (\ref{exp3}) satisfies the cocycle identity, we only need to consider the variation at $\varphi=0$. Let $\varphi_t$ be a smooth path in $\mathcal{K}_\omega$ with $\varphi_0=0, \dot{\varphi}_0=\delta \varphi$. We denote by $\frac{\delta}{\delta\varphi}$ the derivative $\frac{d}{dt}|_{t=0}$. 

We begin with the right-hand side of (\ref{exp3}). The variation of the second term is straightforward:
\begin{align*}
\frac{\delta}{\delta\varphi} &\int_M\Td_{j}(R_\varphi)\frac{1}{(n+1-j)!}\sum^{n-j}_{s=0}\varphi \omega_{\varphi}^s\wedge \omega^{n-j-s}\\
&=\frac{1}{(n-j)!} \int_M \delta\varphi\Td_{j}(R)\omega^{n-j}=\frac{1}{(n-j)!} \int_M \delta\varphi \,\widetilde{t}_j\, \omega^{n}
\end{align*}
where $\widetilde{t}_j$ is the function defined by
\begin{equation*}
\widetilde{t}_j:=\frac{\Td_j(T_J M,g_{(\omega,J)})\omega^{n-j}}{\omega^n}.
\end{equation*}

For the first term, we express the Todd polynomial in terms of the Chern character polynomials as follows:
\begin{equation*}
\Td_j=\sum_{k=(k_1,\cdots,k_r)}a_j^{k_1,\cdots,k_r}\ch_{k_1}\cdots\ch_{k_r}.
\end{equation*}
Then by the property of Bott--Chern forms \cite[Proposition 1.2 (3)]{tian1999bott}, it suffices to consider the variation of the following form:
\begin{equation*}
\frac{\delta}{\delta\varphi} \int_M i\ch_{k_1}(R)\cdots BC(\ch_{k_m};g_{\varphi},g)\cdots\ch_{k_r}(R)\frac{\omega^{n+1-j}}{(n+1-j)!}.
\end{equation*}
This can again be computed using a property of Bott--Chern forms \cite[Proposition 1.2 (4)]{tian1999bott}:
\begin{align*}
&\int_M i\ch_{k_1}\cdots \frac{1}{(k_m-1)!}\Tr \left[ (iR)^{k_m-1}i\dot{g_\varphi}g^{-1} \right] \cdots\ch_{k_r}\frac{\omega^{n+1-j}}{(n+1-j)!}\\
&=-\frac{1}{(n-j)!}\int_M \langle \partial\bar{\partial}\delta\varphi, \widetilde{\ell}^k_m(M, \omega)^\flat\rangle_g \omega^n,
\end{align*}
where $\partial\bar{\partial}\delta\varphi$ denotes the tensor $\frac{\partial^2 \delta\varphi }{\partial z^p\partial \bar{z}^q}dz^p\otimes d\bar{z}^q$ and $\widetilde{\ell}^k_m(M, \omega)\in\Gamma(M,\End T_J M)$ denotes the $\End T_J M$-section defined as follows:
\begin{equation*}
\widetilde{\ell}^k_m(M, \omega):=\frac{1}{n+1-j}\frac{\ch_{k_1}(R)\cdots \frac{1}{(k_m-1)!}(iR)^{k_m-1}\cdots\ch_{k_r}(R)\omega^{n+1-j}}{\omega^n}.
\end{equation*}
By taking the formal adjoints of $\partial$ and $\bar{\partial}$ with respect to $g$, we obtain
\begin{align*}
&\frac{\delta}{\delta\varphi} \int_M iBC(\Td_j;g_{\varphi},g)\frac{\omega^{n+1-j}}{(n+1-j)!}\\
&=\sum_{k=(k_1,\cdots,k_r)}a_j^{k_1,\cdots,k_r} \frac{1}{(n-j)!}\int_M \delta\varphi (-\sum_m \partial^* \bar{\partial}^*\widetilde{\ell}^k_m(M, \omega)^\flat)\omega^n\\
&=:\frac{1}{(n-j)!}\int_M \delta\varphi \,\widetilde{\ell}_j\, \omega^n,
\end{align*}
where $\widetilde{\ell}_j$ is a smooth function. 

\begin{definition}
We define
\begin{equation*}
\widetilde{Z}_j(M,\omega) := \widetilde{t}_j+\widetilde{\ell}_j\in C^\infty(M).
\end{equation*}
Then we have
\begin{align}\label{quilvar}
&\frac{\delta}{\delta\varphi} \int_M iBC(\Td_j;g_{\varphi},g)\frac{\omega^{n+1-j}}{(n+1-j)!}+\Td_{j}(R_\varphi)\frac{1}{(n+1-j)!}\sum^{n-j}_{s=0}\varphi \omega_{\varphi}^s\wedge \omega^{n-j-s}\nonumber\\
&=\frac{1}{(n-j)!}\int_M \delta\varphi \,\widetilde{Z}_j(M,\omega)\, \omega^n
\end{align}
and
\begin{equation*}
\int_M \widetilde{Z}_j(M,\omega) \omega^n = \int_M \Td_j(M) [\omega]^{n-j}.
\end{equation*}
\end{definition}

Note that we intentionally used notation identical to that of \cite{dervan2023universal}, to make it clear that the equations $\widetilde{Z}_j(M,\omega)\equiv \mathrm{constant}$ are examples of \textit{Z--critical equations} in the sense of \cite[Definition 2.4]{dervan2023universal}. More precisely, the equation  $\widetilde{Z}_j(M,\omega)\equiv \mathrm{constant}$ is a $Z$--critical equation associated with the \textit{central charge}
\begin{equation*}
Z(M,[\omega])=i\int_M [\omega]^{n}  -\int_M \Td_j(M)[\omega]^{n-j}.
\end{equation*}
Basically, we have shown that the critical point equation for the functional $i\widetilde{S}_\phi$ in \cite[Remark~2]{eum2025partition} is a $Z$--critical equation with central charge given by the invariant polynomial $\phi$.

We now compute the variation of the left-hand side of (\ref{exp3}). It is immediate that
\begin{equation}\label{TYZ}
\frac{\delta}{\delta\varphi} \int_0^1 dt\int_M \dot{\varphi_t}(a_{j}(g_t)-\Delta_t a_{j-1}(g_t))\frac{\omega_t^n}{n!}=\int_M \delta\varphi(a_{j}(g)-\Delta a_{j-1}(g))\frac{\omega^n}{n!}.
\end{equation}
The variation of $\beta_{j-1}(g_\varphi)$ is obtained by subtracting (\ref{TYZ}) from (\ref{quilvar}); this gives
\begin{equation}\label{torvar}
-\frac{\delta}{\delta\varphi} \beta_{j-1}(g_\varphi)=\int_M \delta\varphi\, b_{j-1}(g) \,\frac{\omega^n}{n!}
\end{equation}
for some function $b_{j-1}$.
That is:
\begin{definition}
Define $b_{j-1}$ for each $j\geq 1$ by
\begin{equation*}
\frac{n!}{(n-j)!}\widetilde{Z}_j(M,\omega)=a_j(g)-\Delta a_{j-1}(g)+b_{j-1}(g).
\end{equation*}
\end{definition}
By the explicit formulas of $a_j$ \cite{lu2000lower} and $\beta_j$ \cite{vasserot1989asymptotics,finski2018full} for small values of $j$, we find that
\begin{equation}
\begin{aligned}
\widetilde{Z}_0(M,\omega) &= 1;\\
n\widetilde{Z}_1(M,\omega) &= \frac12 S;\\
n(n-1)\widetilde{Z}_2(M,\omega) &= -\frac{1}{6} \Delta S + \frac{1}{24}(|R|^2-4|\ric|^2+3S^2).\label{expformula}\\
\end{aligned}
\end{equation}
We will show in Appendix \ref{A} how one can verify (\ref{expformula}) directly from the definition of $\widetilde{Z}_j$. We will give a moment map interpretation of $\widetilde{Z}_j(M,\omega)$ in Section \ref{momentproof}. We also note that the equation $\widetilde{t}_j\equiv \mathrm{constant}$ was considered in \cite{leung1998bando}.

\subsection{Proof of the moment map equations}\label{momentproof}

We now fix $(M,\omega)$ and a prequantization $P=(L,h,\nabla)$, and consider the universal family $\pi : \mathcal{X}\rightarrow \mathcal{J}_{int}$ introduced in Section \ref{spaceof}. Let $\lambda^{(k)}$ be the determinant line bundle on $\mathcal{J}_{int}$ with $\xi=\widetilde{L}^k$. Then $\mathcal{G}$ acts on $\lambda^{(k)}$ by pullback of sections, as in (\ref{pullback}).

We first verify the conditions mentioned at the end of Section \ref{equivline}.
\begin{proposition}
The determinant line bundle $\lambda^{(k)}\rightarrow \mathcal{J}_{int}$ is $\mathcal{G}$-equivariant and $\mathcal{G}$ acts on $\lambda^{(k)}$ by biholomorphisms preserving the Quillen metric $||\cdot||_{\omega,\widetilde{h}^k}$ constructed from the data $g^M=g_\omega, h^\xi=\widetilde{h}^k$.
\end{proposition}

\begin{proof}
Let $\psi\in\mathcal{G}$. By the definition of $\mathcal{G}$, the action of $\psi$ on $\Gamma(M,L^k)$ intertwines $D^2_J$ and $D^2_{\psi.J}$ and maps the eigenspace of $D^2_J$ onto the eigenspace of $D^2_{\psi.J}$ with the same eigenvalue. Hence, the line bundle $\lambda^{(k)} \rightarrow \mathcal{J}_{int}$ is $\mathcal{G}$-equivariant. 

To show that $\mathcal{G}$ acts on $\lambda^{(k)}$ by biholomorphisms, we prove that $\mathcal{G}$ preserves $\widetilde{\nabla}$ on $E^p$ and $P^b:E\rightarrow K^b$ on $U^b$. By the definition of $\widetilde{L}^k\rightarrow\mathcal{X}$, for a smooth section $e$ of $E^p$ and $A\in T\mathcal{J}_{int}$,
\begin{equation*}
\widetilde{\nabla}_A e(J) = \nabla_{A^H} e(z;J)=\nabla_{0\oplus A}e(z;J)=\Pi_J^{0,p}\left.\frac{d}{dt}\right|_{t=0} e(z;J+tA)=:\delta_A e(J),
\end{equation*}
where we used the notation $z\in M, J\in \mathcal{J}_{int}$ to indicate the dependence of sections on $M,\mathcal{J}_{int}$ and $\Pi_J^{0,p}$ to denote the projection onto the $(0,p)$-part. This implies
\begin{equation*}
\widetilde{\nabla}_A (\psi.e)=\delta_A \psi.(e(\psi^{-1}.z;\psi^{-1}.J))=\psi. \delta_{\psi^{-1}_*A}e=\psi. \widetilde{\nabla}_{\psi^{-1}_*A}e.
\end{equation*}
Thus $\mathcal{G}$ preserves $\widetilde{\nabla}$. We have already shown that $\psi$ maps $K^b_J$ onto $K^b_{\psi.J}$. Again, by the definition of $\mathcal{G}$, $\psi$ preserves the $L^2$ Hermitian product on $E$. Hence, $\psi$ preserves the orthogonal projection operators $P^b$ over $U^b$. 

This proves that ${}^0\nabla^b$ is $\mathcal{G}$-invariant over $U^b$. Since $\mathcal{G}$ acts on $\mathcal{J}_{int}$ by biholomorphisms, this implies that $\mathcal{G}$ acts on $\lambda^{(k)}$ by biholomorphisms. Finally, the spectra of $D^2_J$ and $D^2_{\psi.J}$ are identified by the action of $\psi$, and hence the Ray--Singer analytic torsion is $\mathcal{G}$-invariant. This proves the $\mathcal{G}$-invariance of the Quillen metric.
\end{proof}

Define $\Omega^{(k)}=i({}^1\nabla^{\lambda^{(k)}})^2$ to be $i$ times the curvature form of the Quillen metric on $\lambda^{(k)}$. Using (\ref{moment}), one readily derives the moment map for $\mathcal{G}\curvearrowright (\mathcal{J}_{int},\Omega^{(k)})$.
\begin{definition}
Let
\begin{equation*}
\widetilde{\mu}^{(k)}(f)=\mathcal{L}^{\lambda^{(k)}}(f)-{}^1\nabla^{\lambda^{(k)}}_{\widehat{f}}\in \Gamma(\mathcal{J}_{int},\End{\lambda^{(k)}})\simeq C^\infty(\mathcal{J}_{int},\mathbb{C})
\end{equation*}
for $f\in C^\infty(M)\simeq \mathrm{Lie}(\mathcal{G})$. Then $\langle \mu^{(k)}, f\rangle=i\widetilde{\mu}^{(k)}(f)$ defines a moment map $\mu^{(k)}:\mathcal{J}_{int}\rightarrow C^\infty(M)$ for $\mathcal{G}\curvearrowright (\mathcal{J}_{int},\Omega^{(k)})$, via the identification $C^\infty(M)\simeq \mathrm{Lie}(\mathcal{G})^*$ induced by the $L^2$ product $\langle\cdot,\cdot\rangle$ on $M$.
\end{definition}

Now we compute the asymptotic expansion of $\widetilde{\mu}^{(k)}(f)$ as $k\rightarrow\infty$. Let $\varepsilon>0$. By the spectral gap property of $D^2_{L^k,J}$ \cite[Theorem 1.5.5]{ma2007holomorphic}, with the constants depending smoothly on $J$, for each $J\in \mathcal{J}_{int}$, there is an open neighborhood $J\in\mathcal{U}\subset\mathcal{J}_{int}$ of $J$ such that $\mathrm{Spec}(D_{L^k,J'}^2)\cap [0,\varepsilon]=\{0\}$ and $K(\tilde{L}^k)^\varepsilon_{J'}=H^0(M_{J'},L^k_{J'}), \forall J'\in\mathcal{U}$ for uniform $k\gg 1$. In particular, $\mathcal{U}\subset U^\varepsilon$. We will work on $\mathcal{U}$ and assume $k\gg1$ from now on.

By definition, 
\begin{equation*}
\widetilde{\mu}^{(k)}(f)=\mathcal{L}^{\lambda^{(k)}}(f)-{}^1\nabla^{\lambda^{(k)}}_{\widehat{f}}=\mathcal{L}^{\lambda^{(k)}}(f)-{}^0\nabla^{\lambda^{(k)}}_{\widehat{f}}-\iota_{\widehat{f}}\partial^{\mathcal{J}_{int}}\log \tau_{(k)}^2.
\end{equation*}
First, we compute the $\mathcal{L}^{\lambda^{(k)}}(f)-{}^0\nabla^{\lambda^{(k)}}_{\widehat{f}}$ part. We proceed similarly to \cite{foth2007manifold}. Let $e_1,\cdots,e_d$ be a local orthonormal frame for $K^\varepsilon:=K(\tilde{L}^k)^{\varepsilon}$ on $\mathcal{U}$. That is, $(e_j(J))_j$ is an orthonormal basis of $H^0(M_J,L_J^k)$ for each $J\in\mathcal{U}$. Identify $(e_j)_j$ with their duals via the Hermitian product. Then
\begin{equation*}
\mathcal{L}^{\lambda^{(k)}}(f)-{}^0\nabla^{\lambda^{(k)}}_{\widehat{f}}(e_1\wedge\cdots\wedge e_d)=\sum_{j=1}^d e_1\wedge\cdots\wedge\left( \mathcal{L}^{K^\varepsilon}(f)-\nabla^{K^\varepsilon}_{\widehat{f}}\right)^*e_j\wedge\cdots\wedge e_d
\end{equation*}
and
\begin{align*}
&\left(\mathcal{L}^{K^\varepsilon}(f)-\nabla^{K^\varepsilon}_{\widehat{f}}\right)e_j(J)\\
&=\left(\left.\frac{d}{dt}\right|_{t=0} \exp{(tf)}.e_j-P^\varepsilon \delta_{\widehat{f}} e_j \right) (J)\\
&=\left.\frac{d}{dt}\right|_{t=0} \exp{(tf)}.(e_j(\exp{(-tf)}.z;\exp{(-tf)}.J))-P^\varepsilon e_j(z;\exp{(-tf)}.J)\\
&=\left.\frac{d}{dt}\right|_{t=0} \exp{(tf)}.(e_j(\exp{(-tf)}.z;J))+(1-P^\varepsilon)e_j(z;\exp{(-tf)}.J)\\
&=\mathcal{L}^{L^k}(f)e_j(J)+(1-P^\varepsilon)\delta_{\widehat{f}}e_j(J)\\
&=(\nabla^{L^k}_{X_f}-ikf)e_j(J)+(1-P^\varepsilon)\delta_{\widehat{f}}e_j(J).
\end{align*}
Let $P^{(k)}_f:=(\nabla^{L^k}_{-X_f}+ikf)$ be the \textit{Kostant-Souriau operator} or \textit{prequantization operator} acting on $\Gamma(M,L^k)$. Since $X_f$ is Hamiltonian, $P^{(k)}_f$ is skew--Hermitian.
Then
\begin{align*}
&\mathcal{L}^{\lambda^{(k)}}(f)-{}^0\nabla^{\lambda^{(k)}}_{\widehat{f}}(e_1\wedge\cdots\wedge e_d)\\
&=\sum_{j=1}^d e_1\wedge\cdots\wedge\left( -P^{(k)}_f +(1-P^\varepsilon )\delta_{\widehat{f}}\right)^* e_j\wedge\cdots\wedge e_d\\
&=\sum_{j=1}^d \langle (-P^{(k)}_f +(1-P^\varepsilon )\delta_{\widehat{f}})^*e_j,e_j \rangle e_1\wedge\cdots\wedge e_d\\
&=\Tr[P^\varepsilon P^{(k)}_f](e_1\wedge\cdots\wedge e_d).
\end{align*}
Tuynman's lemma \cite{tuynman1987quantization} identifies the operator $Q^{(k)}_f:=P^\varepsilon P^{(k)}_f$ with the \textit{Toeplitz operator} with symbol $i(kf-\Delta f)$.
\begin{lemma}[Tuynman's lemma, \protect{\cite[Proposition 4.1]{bordemann1991gl}}]
For $e,e'\in H^0(M,L^k)$,
\begin{equation*}
\langle Q^{(k)}_f e,e' \rangle=\langle i(kf-\Delta f)e,e' \rangle,
\end{equation*}
where $\langle\,,\,\rangle$ denotes the $L^2$ product on $\Gamma(M,L^k)$.
\end{lemma}
Hence we have
\begin{align}\label{l2exp}
\Tr[P^\varepsilon P^{(k)}_f]&=\int_M i(kf-\Delta f)\rho_k \frac{\omega^n}{n!}\nonumber\\
&= \int_M i f(k\rho_k-\Delta\rho_k)\frac{\omega^n}{n!}\nonumber\\
&=\sum_{j=0}^\infty k^{n+1-j}\int_M if(a_j-\Delta a_{j-1} )\frac{\omega^n}{n!}
\end{align}
as $k\rightarrow \infty$.
Next, we compute the $-\iota_{\widehat{f}}\partial^{\mathcal{J}_{int}}\log \tau_{(k)}^2$ part. First,
\begin{equation*}
-\iota_{\widehat{f}}\partial\log \tau_{(k)}^2=-\widehat{f}^{1,0}\cdot\log \tau_{(k)}^2= \frac{i}{2}\mathbb{J}\widehat{f}\cdot\log \tau_{(k)}^{2}
\end{equation*}
since $\tau_{(k)}$ is $\mathcal{G}$-invariant. Recall from (\ref{infaction}) that $\widehat{f}_J=\mathcal{L}_{-X_f}J$. Thus, $\mathbb{J}\widehat{f}_J=J\mathcal{L}_{-X_f}J=\mathcal{L}_{-JX_f}J$ by integrability of $J\in\mathcal{J}_{int}$. Therefore,
\begin{equation}\label{integrability}
\mathbb{J}\widehat{f}\cdot\log \tau_{(k)}^2(J)=\left.\frac{d}{dt}\right|_{t=0} \log\tau_{(k)}^2(J_t),
\end{equation}
where $J_t:=\left(\Phi^{t}_{-JX_f}\right)^*J$ and $\Phi^{t}_{-JX_f}$ is the flow of the vector field $-JX_f$. Since the derivatives 
over the base commute with the asymptotics (\ref{fintor}), we take this derivative for each $\beta_{j}$ (we will see that coefficients $\alpha_j$ appearing in front of $k^{n-j}\log k$ play no role in the moment map picture; they yield constant functions over $\mathcal{J}_{int}$). Let $\widetilde{\beta}_j(g_{(\omega,J)})$ denote the local integral density of $\beta_j(g_{(\omega,J)})$. Using $g_{(\omega,J_t)}=\left(\Phi^{t}_{-JX_f}\right)^*g_{((\Phi^{t}_{JX_f})^*\omega,J)}$, we compute
\begin{align*}
&\left.\frac{d}{dt}\right|_{t=0}\beta_j(g_{(\omega,J_t)})\\
&=\left.\frac{d}{dt}\right|_{t=0} \int_M \widetilde{\beta}_j(g_{(\omega,J_t)})\\
&=\left.\frac{d}{dt}\right|_{t=0} \int_M  \left(\Phi^{t}_{-JX_f}\right)^*\widetilde{\beta}_j(g_{((\Phi^{t}_{JX_f})^*\omega,J)})\\
&=\left.\frac{d}{dt}\right|_{t=0} \int_M  \widetilde{\beta}_j(g_{((\Phi^{t}_{JX_f})^*\omega,J)}).
\end{align*}
By (\ref{hamilton}) and Cartan's formula,
\begin{equation*}
\left.\frac{d}{dt}\right|_{t=0}(\Phi^{t}_{JX_f})^*\omega=\mathcal{L}_{JX_f}\omega=2i\partial\bar{\partial}f.
\end{equation*}
Hence
\begin{equation}\label{RS}
-\frac{i}{2}\mathbb{J}\widehat{f}\cdot\beta_j(g_{(\omega,J)})=-\frac{i}{2}\frac{\delta}{\delta \varphi}\beta_j(g_{(\omega_\varphi,J)})=\int_M if\,b_j \,\frac{\omega^n}{n!}
\end{equation}
by (\ref{torvar}), where we set $\delta \varphi=2f$ in the notation of Section \ref{Zcrit}.
Combining (\ref{l2exp}) and (\ref{RS}), we obtain the asymptotic expansion of $\widetilde{\mu}^{(k)}(f)$ as $k\rightarrow\infty$, where only the powers $k^{n+1-j}$ are recorded:
\begin{equation*}
\widetilde{\mu}^{(k)}(f)=\sum_{j=0}^\infty k^{n+1-j}\int_M if(a_j-\Delta a_{j-1}+b_{j-1} )\frac{\omega^n}{n!}.
\end{equation*}
The associated moment map $\mu^{(k)}:\mathcal{J}_{int}\rightarrow C^\infty(M)$ then has the asymptotic expansion, again recording only the powers $k^{n+1-j}$:
\begin{equation*}
\mu^{(k)}(J)=\sum_{j=0}^\infty k^{n+1-j}\left[-a_j\left(g_{(\omega,J)}\right)+\Delta a_{j-1}\left(g_{(\omega,J)}\right)-b_{j-1}\left(g_{(\omega,J)}\right) \right].
\end{equation*}

By the BGS curvature formula (\ref{Qcurv}), we have
\begin{align}\label{qcurvasym}
\Omega^{(k)}&=-\left[\int_{\mathcal{X}/\mathcal{J}_{int}}\Td(TM,g_{(\omega,J)})\ch(\widetilde{L}^k,\widetilde{h}^k)\right]^{(2)}\nonumber\\
&=\sum_{j=0}^{n+1}k^{n+1-j}\left[-\int_{\mathcal{X}/\mathcal{J}_{int}}\Td_j(TM,g_{(\omega,J)})\ch_{n+1-j}(\widetilde{L},\widetilde{h})\right]\nonumber\\
&=:\sum_{j=0}^{n+1}k^{n+1-j} \left[-\Omega_j\right]
\end{align}
for $2$-forms $\Omega_j$ on $\mathcal{J}_{int}$. Finally, by taking the asymptotic expansion of the moment map equation $d\langle \mu^{(k)}(\cdot),f\rangle=\iota_{\widehat{f}}\Omega^{(k)}$ and using the commutativity of differentiation over the base with asymptotic expansion, we deduce our main result, Theorem \ref{thm:mainA}, the statement of which we recall below.
\begin{theorem}\label{main}
The map $\mu_j:\mathcal{J}_{int}\rightarrow C^\infty(M)$ defined by 
\begin{equation*}
\mu_j(J):=a_j\left(g_{(\omega,J)}\right)-\Delta a_{j-1}\left(g_{(\omega,J)}\right)+b_{j-1}\left(g_{(\omega,J)}\right)=\frac{n!}{(n-j)!}\widetilde{Z}_j(M_J,\omega)
\end{equation*}
is a moment map for $\mathcal{G}\curvearrowright (\mathcal{J}_{int},\Omega_j)$, where $\Omega_j\in \mathcal{A}^2(\mathcal{J}_{int})$ is given by the fiber integral
\begin{equation*}
\Omega_j=\int_{\mathcal{X}/\mathcal{J}_{int}}\Td_j(TM,g_{(\omega,J)})\ch_{n+1-j}(\widetilde{L},\widetilde{h}).
\end{equation*}
\end{theorem}
Using the explicit formula for $\widetilde{Z}_j$ in \eqref{expformula} when $j=1$, together with Fujiki's fiber integral formula \cite[Theorem 4.4]{fujiki1992moduli}, it follows that Theorem \ref{main} generalizes the Donaldson--Fujiki moment map picture for scalar curvature in the integrable case. In Section \ref{FU}, we explain how to derive a generalization of Fujiki's fiber integral formula for $j\geq 1$. 

Note that since (\ref{qcurvasym}) contains no terms of order $k^{m}\log k$, the terms in the asymptotic expansion of $\widetilde{\mu}^{(k)}(f)$ arising from the coefficients $\alpha_j$ in (\ref{fintor}) yield constant functions on $\mathcal{J}_{int}$.

\begin{remark}
Writing down explicit expressions for the forms $\Omega_j$ beyond those provided by Theorem \ref{thm:mainB} is difficult, except in the first few cases. Consequently, their positivity properties are unclear. In the setting of the Kempf--Ness theorem, the positivity of the closed $2$-forms is closely related to the convexity of the associated Kempf--Ness functionals; for $\Omega_2$, the second variation of the corresponding functional is given in \cite[Proposition 1]{eum2025partition}. We also note that, in the context of deformation quantization (see Remark \ref{deforquant}), it is natural for the subleading $2$-forms to be neither positive nor nondegenerate.
\end{remark}

\subsection{Comparison with the result of Foth--Uribe}\label{FU}

In \cite{foth2007manifold}, Foth and Uribe used the determinant line bundle $(\lambda^{(k)})^{-1}_J=\det H^0(M_J,L_J^k)$ for large enough $k$ on each open set $\mathcal{U}$ to derive the moment map equation for the trace of the operator $Q^{(k)}_f$ with respect to the action of $\mathcal{G}$. In particular, they used (the dual of) the invariant connection ${}^0\nabla^{\lambda^{(k)}}$ over $\mathcal{U}$ (called the $L^2$ connection in their paper) and characterized the curvature of ${}^0\nabla^{\lambda^{(k)}}$ in terms of the trace of the Toeplitz operator. The asymptotic expansion then follows. 
\begin{theorem}[\protect{\cite[Theorem 1.2, Theorem 2.1]{foth2007manifold}}]
Let $A,B\in T_J\mathcal{J}_{int}$. Then there are full symbol coefficients $s_r(A,B)$ such that as $k\rightarrow\infty$,
\begin{align*}
{}^0\Omega^{(k)}(A,B)&:=i\left( {}^0\nabla^{\lambda^{(k)}}\right)^2(A,B)\\
&=\sum_{r=0}^\infty k^{-r}\operatorname{Tr}T_{s_r(A,B),k}=\sum_{j=0}^{\infty} k^{n-j}\int_M \sum_{r=0}^j s_r(A,B)a_{j-r}(g_{(\omega,J)})\frac{\omega^n}{n!}.
\end{align*}
where $T_{s_r(A,B),k}$ denotes the Toeplitz operator with symbol $s_r(A,B)$ at level $k$. In particular, 
\begin{equation*}
s_0(A,B)=\frac{1}{8}\operatorname{Tr}_{TM}(JAB).   
\end{equation*}
See \cite[Corollary A.7]{foth2007manifold} for the missing factor of $i$. Our normalization of $\omega$ differs by a factor of $2$ from that used there.
\end{theorem}
A similar but more general characterization was also obtained in \cite{ma2023superconnection}. Unlike $\Omega^{(k)}$, ${}^0\Omega^{(k)}$ is defined only locally for large enough $k$. Denote the $k^{n-j}$-coefficient of ${}^0\Omega^{(k)}$ by $-{}^0\Omega_j$, that is, 
\begin{equation*}
-{}^0\Omega_j(A,B):=\int_M \sum_{r=0}^j s_r(A,B)a_{j-r}(g_{(\omega,J)})\frac{\omega^n}{n!}
\end{equation*}
for $A,B\in T_J\mathcal{J}_{int}$. Then, by comparing coefficients, one obtains the moment map property of $J \mapsto a_j(g_{(\omega,J)})-\Delta a_{j-1}(g_{(\omega,J)})$ for $\mathcal{G}\curvearrowright (\mathcal{J}_{int},{}^0\Omega_{j-1})$. Thus, \textit{locally} on $\mathcal{U}$, we are in the situation described in (\ref{compare}), and our result is compatible with the results of Foth and Uribe. Since $\Omega_j={}^0\Omega_{j-1}-i\partial\bar{\partial}\beta_{j-1}$, we obtain a generalization of Fujiki's fiber integral formula, Theorem \ref{thm:mainB}.

\appendix 
\section{Explicit computation of the Z--critical equation}\label{A}

In this appendix, we compute the $Z$--critical equation corresponding to the central charge $\int_M \ch_2(M) [\omega]^{n-2}$. 
First, by standard identities in K\"ahler geometry \cite[Lemma 4.7]{szekelyhidi2014introduction},
\begin{equation*}
\widetilde{t}=\frac{\ch_2(TM,g)\omega^{n-2}}{\omega^n}=\frac{1}{2n(n-1)}\left( |\ric|^2-|R|^2 \right).
\end{equation*}
Second, by the same identities,
\begin{equation*}
\widetilde{\ell}=\frac{1}{n-1}\frac{iR\omega^{n-1}}{\omega^n}=\frac{1}{n(n-1)}R_{p}{}^{q}{}_{j\bar{k}}g^{\bar{k}j}=\frac{1}{n(n-1)}{\ric}_{p}{}^{q}=:\frac{1}{n(n-1)}\widetilde{\ric}\in \Gamma(M,\End T^{1,0}M).
\end{equation*}
Now we have to apply formal adjoints of $\partial$ and $\bar{\partial}$ to $\widetilde{\ric}$. That is, we need to evaluate
\begin{equation*}
\int_M \langle \partial\bar{\partial}f, \widetilde{\ric}^\flat\rangle_g \omega^n=\int_M f \,\partial^* \bar{\partial}^* \widetilde{\ric}^\flat \,\omega^n
\end{equation*}
for an arbitrary function $f$, where $\partial\bar{\partial}f$ denotes the tensor $\frac{\partial^2 f }{\partial z^p\partial \bar{z}^q}dz^p\otimes d\bar{z}^q$. This integration by parts can be done more easily in the language of differential forms. Note that
\begin{equation*}
\langle \partial\bar{\partial}f, \widetilde{\ric}^\flat\rangle_g=\partial_j\bar{\partial}_kf {\ric}_{p}{}^{j}g^{\bar{k}p}=\partial_j\bar{\partial}_kf {\ric}_{p\bar{q}}g^{\bar{q}j}g^{\bar{k}p}=\langle i\partial\bar{\partial}f, \ric \rangle_\omega
\end{equation*}
where the last term denotes the inner product of the Hermitian $(1,1)$-forms $i\partial\bar{\partial}f$ and the Ricci form $\ric=i\ric_{p\bar{q}}dz^p \wedge d\bar{z}^q$
in \cite[Lemma 4.7]{szekelyhidi2014introduction}. Again, by the same identities,
\begin{equation*}
\int_M \langle i\partial\bar{\partial}f, \ric \rangle_\omega \omega^n=\int_M \Delta f S \omega^n -n(n-1)i\partial\bar{\partial}f\wedge\ric\wedge\omega^{n-2}=\int_M  f \Delta S \omega^n,
\end{equation*}
which implies
\begin{equation*}
-\partial^* \bar{\partial}^* \widetilde{\ell}^\flat=\frac{-1}{n(n-1)}\Delta S.
\end{equation*}
Combining these, we obtain 
\begin{equation*}
\widetilde{Z}(M,\omega)=\frac{1}{n(n-1)}(-\Delta S -\frac{1}{2}|R|^2+ \frac{1}{2}|\ric|^2).
\end{equation*}

It follows from \cite[Example 2.6]{dervan2023universal} that the formula for $\cc_1^2$ is
\begin{equation*}
\widetilde{Z}(M,\omega)=\frac{1}{n(n-1)}(-2\Delta S +S^2 -|\ric|^2).
\end{equation*}
We can also compute this directly using the same trick as above. Since $\Td_2=-\frac{1}{12}\ch_2+\frac{1}{8}\cc_1^2$, this recovers (\ref{expformula}).

\bibliographystyle{alpha}
\bibliography{References}

\end{document}